\documentclass{article}
\usepackage{romp_arxiv}

\usepackage[normalem]{ulem}

\usepackage[utf8]{inputenc} 
\usepackage[english]{babel}

 \usepackage{latexsym}
 \usepackage{amsmath}
 \usepackage{amsthm}
 \usepackage{amsfonts}
 \usepackage{amssymb}
 \usepackage{pdflscape}  
\usepackage{graphicx}
\usepackage{blkarray,booktabs,bigstrut}
\usepackage{mathtools}
\usepackage{cite}

\usepackage{xcolor}

\usepackage[T1]{fontenc}
\usepackage[utf8]{inputenc}

\usepackage[unicode=true]{hyperref}

\newcommand{\gen}{\mathrm{gen}}

\def\zz{{\mathbb Z}}
\def\cc{{\mathbb C}}
\def\nn{{\mathbb N}}
\def\dd{{\mathrm d}}

\def\fp{\mathrm{fp}}
\def\i{{\rm i}}
\def\e{{\rm e}}
\def\d{{\rm d}}
\DeclareMathOperator{\diag}{diag}     
\def\cF{{\cal F}}
\def\cG{{\cal G}}
\newcommand{\lrd}{\overset{\leftrightarrow}{\partial}}
\newcommand{\rv}{\mathop{\mathrm{rv}_0}}

\title{ Confluent functions, Laguerre polynomials 
\\ and their (generalized) bilinear integrals }
\author{ Jan Derezi\'nski\thanks{ \hspace{1ex}Supported by the National Science Center of Poland under the 
    grant UMO-2019/35/B/ST1/01651. } \\ Department of Mathematical Methods in Physics, Faculty of Physics, \\
University of Warsaw, Pasteura 5, 02-093 Warszawa, Poland \\ e-mail:  jan.derezinski@fuw.edu.pl \vspace{2ex}
\\
         Christian Ga\ss \thanks{ \hspace{1ex}Supported by the National Science Center of Poland under the 
    grant UMO-2019/35/B/ST1/01651. }
                      \\ Department of Mathematical Methods in Physics, Faculty of Physics, \\
University of Warsaw, Pasteura 5, 02-093 Warszawa, Poland; and \\
Faculty of Physics, University of Vienna \\
Boltzmanngasse 5, A-1090 Vienna, Austria 
\\ e-mail:  christian.gass@fuw.edu.pl \vspace{2ex}
\\
         Joonas Mikael V\"att\"o\thanks{ \hspace{1ex}Supported by the Academy of Finland grant number 340461 "Conformal invariance in planar random geometry" and the Academy of Finland Centre of Excellence Programme grant number 346315 "Finnish centre of excellence in Randomness and STructures (FiRST)". }
                      \\ Department of Mathematics and Systems Analysis,
\protect\\
Aalto University, Otakaari 1, 02150 Espoo, Finlands \\ e-mail:  joonas.vatto@aalto.fi                     }

\begin{document}
\maketitle
\begin{abstract}
     We review properties of confluent functions and the closely related 
Laguerre polynomials, and determine their bilinear integrals.
As is well-known, these integrals are convergent only for a limited 
range of parameters. However, when one uses {\em the generalized integral} 
they can be computed essentially without restricting the parameters. 
This gives the (generalized) Gram matrix  of Laguerre polynomials.
If the parameters are not negative integers, then Laguerre polynomials
are orthogonal, or at least pseudo-orthogonal in the case of 
generalized integrals. For negative integer
parameters, the orthogonality relations are more complicated.
\end{abstract}


\section{Introduction}

Consider the (formal) differential operator
\begin{align}
  \cF_{\alpha}&\coloneqq -z\partial_z^2-(1+\alpha-z)\partial_z.
  \label{eq:confl_eq_expform.}
\end{align}
{\em Kummer's confluent function}
$F_{\theta,\alpha}(z)\coloneqq {}_1F_1(a;c;z)$ is an eigenfunction 
of $\cF_\alpha$                               with
                                  eigenvalue
                                  $\frac{1+\theta+\alpha}{2}$, where
                                  following  \cite{De1,De2,De3}
                                  instead of $a,c$ we prefer to use
                                  the parameters
                                  \begin{align} \alpha\coloneqq c-1, \qquad \theta: =-c+2a.\label{popi}\end{align}
   
Besides Kummer's function, \eqref{eq:confl_eq_expform.} possesses
a second type of distinguished eigenfunction with eigenvalues
   $\frac{1+\theta+\alpha}{2}$, 
 called {\em Tricomi's confluent functions},  which we denote
$U(a;c;z)$ or, in the parameters
\eqref{popi}, $U_{\theta,\alpha}(z)$. Tricomi's functions usually have 
a better behavior near infinity than Kummer's functions.

The main topic of this paper is the computation of bilinear 
integrals of Tricomi functions, 
more precisely, integrals   of the form
\begin{align}   \quad   \int_0^\infty U_{\theta_1,\alpha}(z) 
U_{\theta_2,\alpha}(z)\e^{-z}z^\alpha\d z.
\label{popi1}\end{align}
It is known that  for $|\Re(\alpha)|<1$ \eqref{popi1} can be 
computed explicitly, see e.g. \cite{DFNR}. 
   For $|\Re(\alpha)|\geq1$, 
\eqref{popi1} is divergent --- except for very specific combinations 
of the parameters --- because of the behavior of the
integrand near $0$.    
However, when one applies the so-called {\em generalized integral}, one can
compute \eqref{popi1} for all values of $\alpha$.
For $\alpha\in\zz$, $\alpha\neq0$, one needs to apply the anomalous
generalized integral, and one obtains a much more complicated expression.

For special values of the parameters, Tricomi's functions coincide up to
a coefficient with the well-known Laguerre polynomials.
Their classical definition uses the so-called {\em Rodrigues formula }
\begin{align}
L_n^\alpha(z) \coloneqq \frac{1}{n!}\e^zz^{-\alpha}\partial_z^n
\e^{-z}z^{n+\alpha}, \quad n\in\nn_0,
\end{align}
and applies to all $\alpha\in\cc$. They are, however, orthogonal with respect 
to the positive definite measure $\dd \mu(z) = z^\alpha \e^{-z} \dd z$ 
on $L^2[0,\infty[$ only if $\alpha>-1$.

The orthogonality relations for Laguerre polynomials, valid for $\Re(\alpha)>-1$,
  \begin{align}
\label{eq:scalar_product_laguerre.}
 \int_0^\infty L_m^\alpha(z) L_n^\alpha(z) z^\alpha \e^{-z} \dd z 
 = \frac{\Gamma(1+n+\alpha)}{n!}\delta_{m,n},
\end{align}
belong to the  standard knowledge. Again, using the generalized 
integral, we extend the relations \eqref{eq:scalar_product_laguerre.} 
to all $\alpha\in\cc$. For $-\alpha\not\in\nn$ we still have the same
expression as in \eqref{eq:scalar_product_laguerre.}.
In particular, $L_n^\alpha$ are (pseudo-)orthogonal to one another for 
$m\neq n$. The situation is more complicated for $-\alpha\in\nn$, 
where the generalized integral becomes anomalous.

One can derive the bilinear integrals of Laguerre polynomials from
those of Tricomi's functions. However, we prefer a direct method,
which mimics the usual proof of the orthogonality relations based on
the Rodrigues formula and repeated integration by parts, known from 
standard textbooks. Interestingly, integration by parts
works quite well  also for generalized integrals. 
 This is not obvious, so we dedicate a subsection in 
Appendix \ref{app:gen_int} to the description of integration by 
parts in the context of generalized integrals.

For  $\alpha\not\in[-1,\infty[$ (typically real, preferably
integer),  Laguerre (or other orthogonal) 
polynomials  are of interest in several fields of 
mathematics, see for example \cite{Schur,Sell04,ShoreySinha,KuijlaarsMcLaughlin,Behr}.
Apparently, the value $\alpha=-1$ is especially interesting in
applications \cite{Behr-oral}.

In a recent work \cite{DGR23a}, some of us (with collaborators) used
the generalized integral, a concept that goes back to independent 
considerations of Riesz \cite{Riesz} and Hadamard 
\cite{Hadamard23,Hadamard32}, to give a meaning to integrals over 
classically non-integrable functions.
The generalized integrals of these functions 
appear as higher-dimensional generalizations of Green's functions for Laplacians
perturbed by point interactions \cite{DGR24a}, forming toy models for
various facets of renormalization in quantum field theory.
The generalized integral is 
a linear continuation of the standard integral. It is closely related 
to the extension of homogeneous distributions \cite{Hoermander90,Paycha} 
and the so-called barred integral \cite{Lesch97}. We find it
interesting that replacing the integral 
by the generalized integral, we are able to compute bilinear integrals
of the form \eqref{popi1} and \eqref{eq:scalar_product_laguerre.} for arbitrary complex $\alpha$. 

Our paper is not the first to address generalized orthogonality 
relations of the classical orthogonal polynomials. There is some 
literature on the subject 
\cite{KrallFrink,kwon1,kwon2,kwon3,MortonKrall,Krall_the_other,BL12},  
where generalized orthogonality relations for some $\alpha\in\cc$ with 
$\Re(\alpha)<-1$ are often described by changing the integral measure, 
typically allowing for distributional measures. The generalized integral 
is a natural tool that allows us to treat all $\alpha$  simultaneously. 
We have not seen such a comprehensive description in the literature.

Integrals of expressions containing confluent functions can be found in many
sources. However, most of them involve a single confluent function, such as 
those in \cite{Dattoli1}, where they are derived with the so-called umbral 
technique. Integrals of products of two confluent functions are computed 
e.g. in \cite{DFNR}; they are less frequent in the literature. The identity 
\eqref{telescope1}, whose proof involves a telescoping series, belongs to a 
class of identities  treated by umbral methods in \cite{Dattoli2}.
  
Here is the plan of our paper. Section \ref{sec:bessel} is devoted to
the confluent equation for general parameters. First we recall its
basic theory, following mostly \cite{De1,De2}. In
  particular, we review the definitions of four standard solutions of
  the confluent equation. Three of them typically blow up at infinity,
  and only one of them, the Tricomi function, can be used in
   bilinear (generalized) integrals of the form \eqref{popi1}.
  The main  new result of
this section are the expressions for these integrals for all parameters.

Section \ref{Laguerre polynomials} is devoted to Laguerre
polynomials. Again, we start with a concise exposition of their
theory. Then we compute their bilinear integrals.   We
  use two methods. Our first method is
based on integration by parts, which interestingly works well also in
the case of generalized integrals. Then we give an alternative
derivation: we show how bilinear integrals of Laguerre polynomials
can be obtained as special cases of bilinear integrals of Tricomi
functions obtained in Section \ref{sec:bessel}.

In Appendix \ref{app:gen_int} we give a resum\'e of  properties of
generalized integrals based mostly on \cite{DGR23a}. The appendix
contains  also a new fact not discussed in \cite{DGR23a}: the 
integration by parts property of generalized integrals. Appendix
\ref{app:lemmas} contains proofs of two identities for the
digamma functions and the Pochhammer symbols needed in our analysis.

\section{Confluent equation}
\label{sec:bessel}

In this section we start with an overview of the
  confluent hypergeometric equation and  its standard  solutions.
All the identities of this section until  Subsection \ref{Relationship to the modified Bessel equation}
can be found in one form or another in standard sources,
such as \cite{NIST}. Nevertheless, we believe that our presentation,
influenced by \cite{De1,De2,De3}, is quite different from the usual
one. We stress that the confluent equation has two different but
equivalent forms:
it can be understood as an eigenvalue equation of the ${}_1F_1$ operator
\eqref{eq:confl_eigeneq} and of
the ${}_2F_0$ operator \eqref{eq:confl_eigeneq2}. These two forms are
related to the so-called Lie-algebraic parameters \eqref{lie-alg},
which are helpful in organizing the properties of the confluent equation.

  For convenience, and as we employ them heavily, we recall the definition of the Pochhammer symbol and the harmonic numbers
\begin{equation}
(a)_n = a(a+1)...(a+n-1), \quad H_k(a) = \sum_{j=0}^{k-1}\frac{1}{a+j},
\end{equation}
which admit a useful extension to complex parameters in terms of the Gamma and Digamma functions,
\begin{equation}
	(a)_z = \frac{\Gamma(a+z)}{\Gamma(a)}, \quad H_z(a)=\psi(a+z)-\psi(a), \quad z \in \cc \setminus (-a-\nn_0).
\end{equation}

\subsection{Hypergeometric  ${}_1F_1$ function}
 {\em (Kummer's) confluent function}, also called the
{\em hypergeometric  ${}_1F_1$ function}, is defined by 
the following series convergent in $\cc$:
\begin{align}
{}_1F_1(a;c;z)&\coloneqq \sum_{n=0}^\infty
\frac{(a)_n}{(c)_n}\frac{z^n}{n!}.
\end{align}
It is a confluent form of the ${}_2F_1$ function
\begin{align}
{}_1F_1(a;c;z)
=\lim_{b\to\infty} {}_2F_1\big(a,b;c;\tfrac{z}{b}\big),
\end{align}
which satisfies the {\em confluent quation}
\begin{align} 
(z\partial_z^2+(c-z)\partial_z-a){}_1F_1(a;c;z)=0
\end{align}
and the {\em 1st Kummer identity}
\begin{align} 
{}_1F_1(a;c;z)=\e^z {}_1F_1(c-a;c;-z).
\end{align}
It is often preferable to replace it by the so-called 
{\em Olver's normalized ${}_1F_1$ function}
\begin{align}{}_1{\bf F}_1  (a;c;z)
&\coloneqq \frac{{}_1F_1(a;c;z)}{\Gamma(c)}
=\sum_{n=0}^\infty 
\frac{(a)_n}{\Gamma(c+n)}\frac{z^n}{n!}.
\end{align}

Here are its integral representations:
 for all parameters
\begin{eqnarray} \label{eq:intrep1}
\frac{1}{2\pi \i}\int\limits
_{]-\infty,(0,z)^+,-\infty[} t^{a-c}\e^t(t-z)^{-a}\d t
&=& {}_1{\bf F}_1  (a;c;z),\end{eqnarray}
for $\Re(a)>0,\ \Re (c-a)>0$,
\begin{eqnarray}\frac{1}{\Gamma(a)\Gamma(c-a)}\int\limits_{[1,+\infty[}\e^{\frac{z}{t}}t^{-c}(t-1)^{c-a-1}\d t
&=& {}_1{\bf F}_1   (a;c;z),\end{eqnarray}
and for $\Re (c-a)>0$ and $a\notin\nn$, 
\begin{eqnarray}
\frac{\Gamma(1-a)}{2\pi\i\Gamma(c-a)}\int\limits_{[1,0^+,1]}\e^{\frac{z}{t}}(-t)^{-c}(-t+1)^{c-a-1}\d t
&=& {}_1{\bf F}_1   (a;c;z).
\label{kumme}\end{eqnarray}
The meaning of the contours in \eqref{eq:intrep1} and \eqref{kumme} is 
the following. In \eqref{eq:intrep1}, we start at $-\infty$, go around 
both $0$ and $z$ in the positive sense and go back to $-\infty$. 
Similarly, in \eqref{kumme}, we start from $1$, go around $0$ in the 
positive sense and go back to $1$. 

For integer  $m$ we have  expressions in terms of elementary
functions:
\begin{align}
{}_1F_1(c+m;c;z)&=\frac{z^{1-c}}{(c)_m}\partial_z^m z^{c-1+m}\e^z,\\
{}_1{\bf  F}_1(c+m;c;z)&=\frac{z^{1-c}}{\Gamma(c+m)}\partial_z^m z^{c-1+m}\e^z.\end{align}
Here, to define $\partial_z^m$ for negative $m$ we use 
\begin{align}
 \partial_z^{-1}f(z)\coloneqq \gen\int_0^z f(y)\d y,\label{anoma}
\end{align}
provided that \eqref{anoma} is non-anomalous and that the function is integrable in the generalized 
sense often enough.

\subsection{Hypergeometric ${}_2F_0$ function}

The following function arises  as a  result of a  different
confluence of singularities of the hypergeometric function. It is
defined for $w\in\cc\backslash[0,+\infty[$ by 
\[{}_2F_0(a,b;-;w)\coloneqq \lim_{c\to\infty}{}_2F_1(a,b;c;cw),\]
where $|\arg c-\pi|<\pi-\epsilon$, $\epsilon>0$.
It extends to an analytic function on the universal cover of
$\cc\backslash\{0\}$ 
with a branch point of an infinite order at 0.
It has the following asymptotic expansion:
\[
{}_2F_0(a,b;-;w)\sim\sum_{n=0}^\infty\frac{(a)_n(b)_n}{n!}w^n,
\ |\arg w-\pi|<\pi-\epsilon.
\]
It satisfies the equation
\begin{align}
\big(w^2\partial_w^2 +( -1+(1+a+b)w)\partial_w +ab\big)
{}_2F_0(a,b;-;w)=0,\end{align}
which we will call the {\em${}_2F_0$ equation}.

We have an integral representation for $\Re(a)>0$,
\[\frac{1}{\Gamma(a)}\int_0^\infty
\e^{-\frac{1}{t}}t^{b-a-1}(t-w)^{-b}\d t
= { F}   (a,b;-;w), \ \ w\not\in[0,\infty[,
\]and for $a\notin \nn$,
\[\frac{\Gamma(1-a)}{2\pi\i}\int\limits_{[0,w^+,0]}
\e^{-\frac{1}{t}}t^{b-a-1}(t-w)^{-b}\d t
= { F}   (a,b;-;w), \ \ w\not\in[0,\infty[.
\]

\subsection{Equivalence between the ${}_1F_1$ and ${}_2F_0$ equations}

It is easy to derive the following identity:
\begin{align}
\notag
&(-w)^{a+1}\Big(w^2\partial_w^2 +( -1+(1+a+b)w)\partial_w +ab\Big)
(-w)^{-a} =\ z\partial_z^2+(1+a-b-z)\partial_z-a,
   \label{g8a}
   \end{align}
 where $z=-w^{-1}$, $w=-z^{-1}$. 
Hence the ${}_2F_0$ equation is equivalent to the  ${}_1F_1$ 
equation with the following relation  between the parameters 
\[c=1+a-b,\ \ \ \ b=1+a-c.\]

Because of this equivalence, 
 {\em Tricomi's confluent function} 
\[U(a,c,z)\coloneqq z^{-a}{}_2F_0(a,1+a-c;-;-z^{-1})\]
is one of  solutions of the ${}_1F_1$ equation, 

\subsection{Confluent equation in Lie-algebraic parameters}
Instead of the classical parameters we usually prefer  
the Lie-algebraic parameters $\theta,\alpha$:
\begin{equation}\begin{array}{lll}
\alpha\coloneqq c-1=a-b,&& \theta: =-c+2a=-1+a+b;\\[2ex]
a=\frac{1+\alpha+\theta }{2},& b=\frac{1 -\alpha+\theta}{2},
\ \ \ &
 c=1+\alpha.
\end{array}
\label{lie-alg}\end{equation}

Here are the ${}_1F_1$ and ${}_2F_0$ equations in the Lie-algebraic 
parameters:
\begin{align}
\Big(z\partial_z^2+(1+\alpha-z)\partial_z-\frac{1}{2}(1+\theta +\alpha)\Big)f(z)&=0,\\
\Big(w^2\partial_w^2+\big(-1+(2+\theta)w\big)\partial_w+\big(\tfrac{1+\theta}{2}\big)^2-\big(\tfrac{\alpha}{2}\big)^2\Big)f(w)&=0.
\end{align}

We
introduce the {\em ${}_1F_1$ operator} and the {\em ${}_2F_0$ operator}:
\begin{align}
  \cF_{\alpha}(z,\partial_z)& \coloneqq -z\partial_z^2-(1+\alpha-z)\partial_z
  \\ \label{eq:confl_eq_expform}
  &=-z^{-\alpha}\e^z\partial_z
                              z^{1+\alpha}\e^{-z}\partial_z,\\
  \tilde\cF_\theta(z,\partial_w)&\coloneqq-w^2\partial_w^2-\big(-1+(2+\theta)w\big)\partial_w\\&=-w^{-\theta}\e^{-w^{-1}}\partial_w
  w^{2+\theta}\e^{w^{-1}}\partial_w.
                                  \end{align}
The ${}_1F_1$ equation and the ${}_2F_0$ equation  can be interpreted
as the eigenequation of the corresponding
 operators:
\begin{align}
\label{eq:confl_eigeneq}
\Big(-\cF_{\alpha}(z,\partial_z)-\frac{1}{2}(1+\theta 
  +\alpha)\Big)f(z)&=0,\\\label{eq:confl_eigeneq2}
  \Big(-\tilde\cF_\theta(w,\partial_w) +\big(\tfrac{1+\theta}{2}\big)^2-\big(\tfrac{\alpha}{2}\big)^2\Big)f(w)&=0.
\end{align}

We will treat the ${}_1F_1$ confluent equation as the principal one.

Here are Kummer's and Tricomi's confluent functions in the Lie-algebraic parameters:
\begin{align} 
F_{\theta ,\alpha}(z)&
\coloneqq {}_1{ F}_1\Bigl(\frac{1+\alpha+\theta }{2};1+\alpha;z\Bigr)
,\\
 {\bf F}  _{\theta ,\alpha}(z)&\coloneqq
 {}_1{\bf F}_1  \Bigl(\frac{1+\alpha+\theta }{2};1+\alpha;z\Bigr)=
 \frac{1}{\Gamma(\alpha+1)}F_{\theta ,\alpha}(z),\\
U_{\theta ,\alpha}(z)&\coloneqq U\bigl(\tfrac{1+\alpha+\theta
}{2};1+\alpha;z\bigr) = 
z^{\frac{-1-\alpha-\theta}{2}}{_{2}F_0}\bigl(\tfrac{1+\alpha+\theta
}{2},\tfrac{1-\alpha+\theta }{2};-;-z^{-1}\bigr).
\end{align}

We have four standard solutions of the ${}_1F_1$ equation:
\begin{align}
F_{\theta ,\alpha}(z)
    &=\e^zF_{-\theta ,\alpha}(-z),
    &\sim  1\text{ at }0;\\
z^{-\alpha} F _{\theta ,-\alpha}(z)
    &=z^{-\alpha} \e^z F _{-\theta ,-\alpha}(-z),
    &\sim z^{-\alpha}\text{ at }0;\\
  U_{\theta ,\alpha}(z)
  &=z^{-\alpha} U_{\theta ,-\alpha}(z),&\sim z^{-a}
  \text{ at }+\infty;\label{qua3}\\
\e^z  U_{-\theta ,\alpha}(-z)
&=(-z)^{-\alpha}\e^zU_{-\theta ,-\alpha}(-z),& 
\sim  (-z)^{ b-1}\e^z\text{ at }-\infty.
\end{align}

\subsection{Connection formulas}

The space of solutions of the confluent equation is
2-dimensional. Therefore, the Tricomi function for $\alpha\not\in\zz$
can be expressed in terms of Kummer's functions via the 
following connection formulas, whose derivation can be found in
\cite{NIST} and  \cite{De2}:
\begin{align}
\label{eq:connect_U1}
U_{\theta,\alpha}(z)
&=\frac{\pi}{\sin\pi\alpha}\Bigg(-\frac{{\bf F}  _{\theta ,\alpha}(z)}{\Gamma\left(\frac{1+\theta -\alpha}{2}\right)}    
    +\frac{z^{-\alpha} {\bf F}  _{\theta ,-\alpha}(z)}{\Gamma\left(\frac{1+\theta +\alpha}{2}\right)}\Bigg),
\\ \label{eq:connect_U2}
\e^zU_{-\theta,\alpha}(-z)
&=\frac{\pi}{\sin\pi\alpha}\Bigg(
-\frac{{\bf F}  _{\theta ,\alpha}(z)}{\Gamma\left(\frac{1-\theta -\alpha}{2}\right)}
 +\frac{(-z)^{-\alpha} {\bf F}_{\theta ,-\alpha}(z)}
{\Gamma\left(\frac{1-\theta +\alpha}{2}\right)}
\Bigg).
\end{align}

\subsection{Degenerate case}

The case of integer $\alpha$ is called degenerate.
For $\alpha=m\in\zz$ we have the identity
\begin{align}
\big(\tfrac{1-m+\theta}{2}\big)_m
F_{\theta,m}(z)=z^{-m} F_{\theta,-m}(z).\end{align}
We also have additional  integral representations:
\begin{eqnarray}\frac{1}{2\pi\i}\int\limits
_{[(z,0)^+]}\e^{t}\Big(1-\frac{z}{t}\Big)^{-a}
t^{-m-1}\d t&=&
 {\bf F}  _{-1+2a-m ,m}(z),\\[4ex]
\frac{1}{2\pi\i}\int\limits_{[(0,1)^+]}
\exp\Big({\frac{z}{t}}\Big)(1-t)^{-a}t^{-m-1}\d t
&=&
 z^{-m} {\bf F}  _{-1+2a+m ,-m}(z)
.\end{eqnarray}
There are also corresponding  generating functions
\begin{eqnarray}
\e^{t}\big(1-\tfrac{z}{t}\big)^{-a}&=&
\sum_{m\in\zz}t^m {\bf F}  _{-1+2a-m,m}(z),\\
\exp\big({\tfrac{z}{t}}\big)(1-t)^{-a}&=&
\sum_{m\in\zz}t^m { z^{-m}{\bf F}  _{-1+2a+m,-m}(z)}.\end{eqnarray}

The Tricomi function for $\alpha=m\in\zz$ has a logarithmic singularity:
\begin{align} 
U_{\theta,m}(z)=&
\frac{(-1)^{m+1}}{\Gamma(\frac{1+\theta-m}{2})}\Bigg(
\sum_{k=1}^{m}(-1)^{k-1}\frac{(k-1)!\big(\frac{1+m+\theta}{2}\big)_{-k}}{(m-k)!}z^{-k}\\
 &
 +\sum_{j=0}^\infty\frac{(\frac{1+m+\theta}{2})_j}{(m+j)!j!}\Big(\ln(z)
 +\psi\big(\tfrac{1+m+\theta}{2}+j\big)-\psi(j+1)-\psi(j+m+1)\Big)z^j
\Bigg).\notag
\end{align}

\subsection{Relationship to the modified Bessel equation}
\label{Relationship to the modified Bessel equation}
 The confluent equation for $\theta=0$ is equivalent to the 
 modified Bessel equation (hence also to the Bessel equation): 
 \begin{align}\notag&
 2z^{\nu-1}\e^{-\frac{z}{2}}
 \Big(z\partial_z^2+(1+2\nu-z)\partial_z-\tfrac12-\nu\Big)
 z^{-\nu}\e^{\frac{z}{2}}
 \\ =\
&\partial_r^2+\frac1r\partial_r-1-\frac{\nu^2}{r^2},\qquad z=2r,
\quad \alpha=2\nu.
\end{align}
The standard solutions of the confluent equation can be expressed 
in terms of standard solutions of the modified Bessel equation, 
that is, the modified Bessel function and the Macdonald function:
\begin{align}
  I_\nu (r)
&=\frac{1}{\Gamma(\nu +1)}\Big(\frac{r}{2}\Big)^\nu \e^{- r}
                                                                                       {}_1F_1\Big(\nu +{\frac12};2\nu +1;2 r\Big)\\ \notag
&=\frac{\Gamma(2\nu +1)}{\Gamma(\nu +1)}\Big(\frac{r}{2}\Big)^\nu \e^{- r}
                   {\bf F}_{0,2\nu}(2 r)\\
K_\nu (r)&=\sqrt{\frac{\pi}{2r}}
           \e^{-r}{}_2F_0\Big(\frac12+\nu ,\frac12-\nu ;-;-\frac{1}{2r}\Big)\\ \notag
  &=\sqrt{\pi}(2r)^\nu\e^{-r}U_{0,2\nu}(2r).
\end{align}

\subsection{Bilinear integrals}
\label{Bilinear integrals}

Let us start with the usual bilinear integrals of Tricomi's
functions. The following identities are known, and can be found e.g. in
\cite{DFNR} (where instead of Kummer's and Tricomi's confluent functions
the equivalent  Whittaker functions of the first and
second kind are used). Since our paper contains many comparably 
long formulas, it is convenient to introduce the notation 
\begin{align}
  f(x,y) \pm (x \leftrightarrow y) \coloneqq f(x,y) \pm f(y,x)
 \end{align}
 for some expression $f(x,y)$.
\begin{theorem}{Theorem}
For $|\Re(\alpha)| < 1$, the following identities hold:
\begin{align}
\label{eq:Uint_1}
  & \quad \int_0^\infty U_{\theta_1,\alpha}(z) 
  U_{\theta_2,\alpha}(z)\e^{-z}z^\alpha\d z \\ \notag 
  &  = \frac{2\pi}{(\theta_1 - \theta_2)\sin \pi \alpha}\Bigg( \frac{1}{\Gamma(\frac{1+\theta_1-\alpha}{2})\Gamma(\frac{1+\theta_2+\alpha}{2})} - (\theta_1 \leftrightarrow \theta_2) \Bigg),
  \quad \theta_1\neq\theta_2,
\end{align}
and 
\begin{align}
  \label{eq:Uint_2}
  \int_0^\infty  U_{\theta,\alpha}(z)^2
  \e^{-z}z^\alpha\d z
  &= \frac{\pi}{\sin \pi \alpha}
  \Bigg( \frac{\psi(\frac{1+\theta+\alpha}{2})
  -\psi(\frac{1+\theta-\alpha}{2})
  }{\Gamma(\frac{1+\theta+\alpha}{2})
  \Gamma(\frac{1+\theta-\alpha}{2})} \Bigg). 
\end{align}
In the special case $\alpha=0$, we have
\begin{align}
\label{eq:Uint_1_0}
  &\quad   \int_0^\infty U_{\theta_1,0}(z) 
  U_{\theta_2,0}(z)\e^{-z}\d z
 =  \frac{2}{\theta_1 - \theta_2}
  \frac{\psi\big(\tfrac{1+\theta_1}{2}\big)-\psi\big(\tfrac{1+\theta_2}{2}\big)}{\Gamma(\frac{1+\theta_1}{2})\Gamma(\frac{1+\theta_2}{2})} ,
  \quad \theta_1\neq\theta_2,
\end{align}
and 
\begin{align}
  \label{eq:Uint_2_0}
  \int_0^\infty  U_{\theta,0}(z)^2
  \e^{-z}\d z
  &= \frac{\psi'\big(\tfrac{1+\theta}{2}\big)
  }{\Gamma\big(\frac{1+\theta}{2}\big)^2}. 
\end{align}

\end{theorem}
                                  
\begin{proof}
 By \eqref{eq:confl_eigeneq} and \eqref{eq:confl_eq_expform}, 
 we have 
 \begin{align}
 &\quad \frac{\theta_1-\theta_2}{2}
    \int_0^\infty U_{\theta_1,\alpha}(z) 
  U_{\theta_2,\alpha}(z)\e^{-z}z^\alpha\d z  
  \\ \notag 
&=    \int_0^\infty \big( \partial_z
        z^{1+\alpha}\e^{-z}\partial_z
    U_{\theta_1,\alpha}(z) \big)
    U_{\theta_2,\alpha}(z)
    -U_{\theta_1,\alpha}(z) 
    \partial_z z^{1+\alpha}\e^{-z}\partial_z 
    U_{\theta_2,\alpha}(z)\d z 
 \\ \notag 
&=    \lim_{z\downarrow0} \Big( 
 z^{1+\alpha}\e^{-z} \Big(
 U_{\theta_1,\alpha}(z) 
    \lrd_z   U_{\theta_2,\alpha}(z)
    \Big),
 \end{align}
 where the left-right derivative is defined by 
 $ f\lrd g \coloneqq  f \partial g - (\partial f) g$.
Then, using $|\Re(\alpha)|<1$, the connection formula 
\eqref{eq:connect_U1} and $\tfrac{\pi z}{\sin \pi z}
= \Gamma(1+z) \Gamma(1-z)$, we obtain \eqref{eq:Uint_1}.
The formulas  \eqref{eq:Uint_2}, \eqref{eq:Uint_1_0} and \eqref{eq:Uint_2_0} 
are obtained by applying the rule of de l'H\^opital.
\end{proof}

The formulas  \eqref{eq:Uint_1} and \eqref{eq:Uint_2} remain true for 
$\alpha\in\cc\setminus\zz$ if the integral is replaced by the generalized 
integral. For $\alpha\in\zz$, the generalized integral is anomalous. 
It can be computed via dimensional regulatization:

 \begin{theorem}{Theorem}
\label{prop:anomalous_tricomi}
 Let $\alpha \in \zz$. Then, the following identities hold:
\begin{align}
\label{eq:gen_Uint_1}
  &\quad
  \;\gen \int_0^\infty U_{\theta_1,\alpha}(z) 
    U_{\theta_2,\alpha}(z)\e^{-z}z^\alpha\d z \\
    &=  \frac{(-1)^\alpha}{\theta_1-\theta_2}
    \Bigg(\frac{\psi(\frac{1+\theta_1+\alpha}{2})
    +\psi(\frac{1+\theta_1-\alpha}{2})
    }{\Gamma(\frac{1+\theta_1-|\alpha|}{2})
    \Gamma(\frac{1+\theta_2+|\alpha|}{2})}
      - (\theta_1 \leftrightarrow \theta_2) \Bigg)
      \notag\\
&+\frac{(-1)^\alpha}{\Gamma(\frac{1+\theta_1+|\alpha|}{2})
\Gamma(\frac{1+\theta_2+|\alpha|}{2})}
              \sum_{k=0}^{|\alpha|-1}
\big(\tfrac{1+\theta_1-|\alpha|}{2}\big)_k
\big(\tfrac{3+\theta_2-|\alpha|}{2}+k\big)_{|\alpha|-1-k}
\notag               \\&\times                                          
\Big(-\psi(|\alpha|-k)-\psi(k+1)
+\tfrac12H_k\big(\tfrac{1+\theta_1-|\alpha|}{2}\big)
- \tfrac12H_{|\alpha|-1-k}\big(\tfrac{3+\theta_2-|\alpha|}{2}+k\big)\Big);\notag
\end{align}
\begin{align}
\label{eq:gen_Uint_2}
 &\quad\gen\int_0^\infty U_{\theta,\alpha}(z)^2 \e^{-z}z^\alpha\d z 
  \\\notag  &= \frac{ (-1)^\alpha}{2
  \Gamma\big(\tfrac{1+\theta+
              \alpha}{2}\big)\Gamma\big(\tfrac{1+\theta-\alpha}{2}\big) }
\Bigg(  
    \psi\big(\tfrac{1+\theta+|\alpha|}{2}\big)^2
    -\psi\big(\tfrac{1+\theta-|\alpha|}{2}\big)^2
   + \psi'\big(\tfrac{1+\theta+\alpha}{2}\big)
              +\psi'\big(\tfrac{1+\theta-\alpha}{2}\big)  \\
& \hspace{18ex}  -2    \sum_{k=0}^{|\alpha|-1} \frac{\psi(|\alpha|-k) +\psi(k+1)
    }{ \frac{1+\theta-|\alpha|}{2}+k}\Bigg).\notag
\end{align}
\end{theorem}

\begin{proof}
The generalized integral \eqref{eq:gen_Uint_1} can be computed 
by dimensional regularization as outlined in Appendix 
\ref{app:gen_int}, to which we refer the reader for details on the general scheme.
By \eqref{qua3},
the integrand is symmetric under 
$\alpha\leftrightarrow-\alpha$. Hence, it is sufficient to 
determine the generalized integral for $-\alpha\in\nn$.

Actually, it is convenient to pull out a prefactor and determine 
the generalized integral over the auxiliary function 
\begin{align}
 f(\alpha,z) \coloneqq \Gamma\big(\tfrac{1+\theta_1-\alpha}{2} \big)
\Gamma\big(\tfrac{1+\theta_2-\alpha}{2} \big)
U_{\theta_1,\alpha}(z) 
    U_{\theta_2,\alpha}(z)\e^{-z}z^\alpha.
\end{align}
By \eqref{eq:Uint_1},
\begin{align}
  \int_0^\infty f(\alpha,z)\d z=\frac{2\pi}{(\theta_1-\theta_2)\sin\pi\alpha}\Bigg(\frac{\Gamma(\frac{1+\theta_2-\alpha}{2})}{\Gamma(\frac{1+\theta_2+\alpha}{2})}-\frac{\Gamma(\frac{1+\theta_1-\alpha}{2})}{\Gamma(\frac{1+\theta_1+\alpha}{2})}\Bigg).\end{align}
For $m\in\nn$, one then finds for the finite part (cf. Appendix \ref{app:gen_int})
\begin{align}
\label{eq:gen_Uint_fp}
& \fp \int_0^\infty f(\alpha,z) \dd z \Big|_{\alpha=-m}\\
 =& \frac{ (-1)^m}{\theta_1-\theta_2}
 \Bigg( \big(\tfrac{1+\theta_1-m}{2}\big)_{m} 
 \Big(\psi\big(\tfrac{1+\theta_1+ m}{2}\big)
 +\psi\big(\tfrac{1+\theta_1 - m}{2}\big)\Big)
 - (\theta_1 \leftrightarrow \theta_2)\Bigg).
\end{align}
Let us write
\begin{align}
\label{eq:tricomi_f_expansion}
 f(\alpha,z) \coloneqq  z^\alpha\Big(\Gamma\big(\tfrac{1+\theta_1-\alpha}{2} \big)
U_{\theta_1,\alpha}(z)\Big)
\Big(\Gamma\big(\tfrac{1+\theta_2-\alpha}{2} \big) 
    U_{\theta_2,\alpha}(z)\e^{-z}\Big)
\end{align}

Recalling that $\alpha<0$ one sees that negative powers  
in \eqref{eq:tricomi_f_expansion} come from one source: 
the two terms in the brackets can be rewritten using the 
connection formulas \eqref{eq:connect_U1} and \eqref{eq:connect_U2}. 
This yields a sum of four terms, but three of them have at least a power 
$z^{-\alpha}$. The only term which contains a singularity at $z=0$ is 
the product of 
\begin{align}
-\frac{\pi}{\sin\pi\alpha}{\bf F}_{\theta_1,\alpha}(z)
=\sum_{k=0}^\infty\frac{(-1)^k\big(\frac{1+\theta_1-\alpha}{2}\big)_kz^k\Gamma(-\alpha-k)}{k!},
\end{align}
and 
\begin{align} 
 -\frac{\pi}{\sin\pi\alpha}{\bf F}_{-\theta_2,\alpha}(-z)
=\sum_{k=0}^\infty\frac{\big(\frac{1-\theta_2-\alpha}{2}\big)_kz^k\Gamma(-\alpha-k)}{k!}.
\end{align}
Therefore, if $p\leq |\alpha|=-\alpha$, the coefficient of
$f(\alpha,z)$ at $z^{\alpha+p}$ is
\begin{align}
 f_{p}(\alpha) &= \sum_{k=0}^{p} \frac{(-1)^k
 \big(\tfrac{1+\theta_1+\alpha}{2}\big)_{k} 
\big(\tfrac{1-\theta_2+\alpha}{2}\big)_{p-k}\Gamma(-\alpha-k)\Gamma(-\alpha-p+k)}{k!\, (p-k)!}
                 \\
  &= \sum_{k=0}^{p} \frac{(-1)^{p} 
 \big(\tfrac{1+\theta_1+\alpha}{2}\big)_{k} 
\big(\tfrac{1+\theta_2-\alpha}{2}-p+k\big)_{p-k}\Gamma(-\alpha-k)\Gamma(-\alpha-p+k)}{k!\, (p-k)!}
,
\end{align}
where we used $(z)_k=(-1)^k(1-z-k)_k$ in the last step.
Finally, we use
\begin{align}
\gen\int_0^\infty f(-m,z)=\fp\int_0^\infty f(\alpha,z)\d
z\Big|_{\alpha=-m}-
\partial_\alpha f_{m-1}(\alpha)\Big|_{\alpha=-m}.
\end{align}

We next prove \eqref{eq:gen_Uint_2}.
 Applying the rule of de l'H\^opital to the first line of the right
 hand side of \eqref{eq:gen_Uint_1} we obtain
 \begin{equation}
   \frac{ (-1)^\alpha}{2
  \Gamma\big(\tfrac{1+\theta+
              \alpha}{2}\big)\Gamma\big(\tfrac{1+\theta-\alpha}{2}\big) }
\Bigg(  
    \psi\big(\tfrac{1+\theta+|\alpha|}{2}\big)^2
    -\psi\big(\tfrac{1+\theta-|\alpha|}{2}\big)^2
   + \psi'\big(\tfrac{1+\theta+\alpha}{2}\big)
              +\psi'\big(\tfrac{1+\theta-\alpha}{2}\big)  \Bigg).\end{equation}
Next
we use 
 \begin{align}
  \big(\tfrac{1+\theta-|\alpha|}{2}\big)_k
    \big(\tfrac{3+\theta-|\alpha|}{2}+k\big)_{|\alpha|-1-k}
    =    \frac{\big(\tfrac{1+\theta-|\alpha|}{2}\big)_{|\alpha|}
    }{ \frac{1+\theta-|\alpha|}{2}+k }
 \end{align}
to transform the last two lines of the right hand side of
\eqref{eq:gen_Uint_1} into
\begin{align}
  \frac{ (-1)^\alpha}{
  \Gamma\big(\tfrac{1+\theta+
              \alpha}{2}\big)\Gamma\big(\tfrac{1+\theta-\alpha}{2}\big) }
    &\sum_{k=0}^{|\alpha|-1} \frac{1}{ \frac{1+\theta-|\alpha|}{2}+k} \Bigg( 
    -\psi(|\alpha|-k) -\psi(k+1) \\ \notag &
    +\frac12H_k\big(\tfrac{1+\theta-|\alpha|}{2}\big)
    -\frac12H_{|\alpha|-1-k}\big(\tfrac{3+\theta-|\alpha|}{2}+k\big)
    \Bigg)
  .\notag
\end{align}
To simplify this sum, we use the following identity that can 
be proven inductively: 
\begin{equation}\label{induction}
  \sum_{k=0}^{m-1}\frac{H_k(z)}{z+k}=\frac12\big(H_m'(z)+H_m(z)^2\big).\end{equation} 
Namely, using \eqref{induction} we obtain
\begin{align}
  &  \sum_{k=0}^{|\alpha|-1}
\frac{H_k\big(\tfrac{1+\theta-|\alpha|}{2}\big)-H_{|\alpha|-1-k}\big(\tfrac{3+\theta-|\alpha|}{2}+k\big) }{ \frac{1+\theta-|\alpha|}{2}+k}
  \\
=&  \sum_{k=0}^{|\alpha|-1}
\Bigg(\frac{2H_k\big(\tfrac{1+\theta-|\alpha|}{2}\big)-H_{|\alpha|}\big(\tfrac{1+\theta-|\alpha|}{2}\big) }{ \frac{1+\theta-|\alpha|}{2}+k}+\frac{1}{(\frac{1+\theta-|\alpha|}{2}  +  k)^2}\Bigg)
   \notag\\
  =&H_{|\alpha|}'\big(\tfrac{1+\theta-|\alpha|}{2}\big)+H_{|\alpha|}\big(\tfrac{1+\theta-|\alpha|}{2}\big)^2-H_{|\alpha|}\big(\tfrac{1+\theta-|\alpha|}{2}\big)^2-H_{|\alpha|}'\big(\tfrac{1+\theta-|\alpha|}{2}\big)\quad=\quad0.\notag
\end{align} 
\end{proof}

Note that for special choices of the parameters $\theta_1$, $\theta_2$, 
$\theta$, the right-hand sides of \eqref{eq:Uint_1}, \eqref{eq:Uint_2}, 
\eqref{eq:Uint_1_0}, \eqref{eq:Uint_2_0}, \eqref{eq:gen_Uint_1} and 
\eqref{eq:gen_Uint_2} may contain terms of the form $0$ divided 
by $0$. These special choices of parameters include the case when the 
Tricomi functions reduce to (multiples of) Laguerre polynomials.  
We therefore treat this case separately in the following section. 
We will  show that the singularities are merely 
apparent and that the right-hand sides of the mentioned equations 
also make sense for the described special choice of parameters.

\section{Laguerre polynomials}
\label{Laguerre polynomials}

\subsection{Basic properties}

We start with a concise overview of basic properties of Laguerre 
polynomials. They can be found in various sources 
such as \cite{NIST,De1}.

The Laguerre polynomials are traditionally defined by a  
Rodrigues-type formula: 
\begin{align}
\label{eq:Rodriguez}
L_n^\alpha(z):&=\frac{1}{n!}\e^zz^{-\alpha}\partial_z^n
  \e^{-z}z^{n+\alpha}\\
&=       (-1)^n\sum_{k=0}^n\frac{(-\alpha-n)_{n-k}
                                       z^{k}}{k! (n-k)!}.   
\end{align}
which is meaningful for all complex $\alpha\in\cc$. For real 
$\alpha>-1$, the Laguerre polynomials are one of the sets of 
classical orthogonal polynomials. Their natural interval is $[0,\infty[$ 
and the weight is $w(z) = \e^{-z} z^\alpha$.

Up to a normalization, Laguerre polynomials are special cases of 
confluent functions:
\begin{align}
L_n^\alpha(z)&=\frac{(1+\alpha)_n}{n!}F(-n;1+\alpha;z)\\
&=\frac{\Gamma(1+\alpha+n)}{n!}{\bf F}_{-1-\alpha-2n,\alpha}(z)
\\&
 =\frac{(-z)^n}{n!}{}_2F_0(-n, -\alpha-n;-;-z^{-1})
 \\&= \frac{(-1)^n}{n!}U _{-1-\alpha-2n,\alpha}(z).\label{triccomi}
 \end{align}
Their generating functions are:
\begin{align}
\e^{-tz}(1+t)^{\alpha}
=\sum\limits_{n=0}^\infty t^n L_n^{\alpha-n}(z),\\
(1-t)^{-\alpha-1}
\exp{\frac{tz}{t-1}}=\sum\limits_{n=0}^\infty t^nL_n^{\alpha}(z).
\end{align}
They have the integral representation
\begin{align}
L_n^\alpha(z)&=\frac{1}{2\pi\i}\int\limits_{[0^+]}
\e^{-tz}(1+t)^{\alpha+n}t^{-n-1}\d t.
\end{align}
 The 
value at 0 and behavior at $\infty$:
\begin{align}
L_n^\alpha(0)=\frac{(\alpha+1)_n}{n!},\ \  \ \ 
\lim\limits_{z\to\infty}\frac{L_n^\alpha(z)}{z^n}=\frac{(-1)^n}{n!}.
\end{align}

For $\alpha\in \nn$ with $\alpha\leq n$, we have \cite{Szego39}
\begin{align}
\label{eq:special_alpha_Laguerre}
L_n^{-\alpha}(z)&= \frac{(n-\alpha)!}{n!} (-z)^{\alpha}
L_{n-\alpha}^{\alpha}(z),\quad \alpha\in\nn,\;\alpha\leq n. 
\end{align} 
Note that in this case, $L_n^{-\alpha}(z)$ has a zero of order 
$\alpha$ at the origin.

\subsection{Bilinear integrals}

The  bilinear identity for standard integrals, which for real
$\alpha>-1$ expresses the orthogonality relations, is well-known:
\begin{theorem}{Theorem}Let $\Re (\alpha) >-1$. Then
  \begin{align}
\label{eq:scalar_product_laguerre}
 \int_0^\infty L_m^\alpha(z) L_n^\alpha(z) z^\alpha \e^{-z} \dd z 
 = \frac{\Gamma(1+n+\alpha)}{n!}\delta_{m,n}.
\end{align}
\end{theorem}

\proof For completeness we sketch the proof.
Assume without loss of generality that $m\geq n$. Applying the 
Rodrigues formula \eqref{eq:Rodriguez}
for $L_m^\alpha$, and then integrating $m$ times by parts and
using $\Re(\alpha)>-1$, one obtains\begin{align}
                                   \int_0^\infty L_m^\alpha(z) L_n^\alpha(z) z^\alpha \e^{-z} \dd z&=
                                                                                                     \int_0^\infty\Big(\partial_z\e^{-z}z^{m+\alpha}\Big)L_n^\alpha(z)\dd z\\
 &= \frac{(-1)^m}{m!} \int_0^\infty  
   \e^{-z}z^{m+\alpha}  \partial_z^m  L_n^\alpha(z)  \dd z.
\end{align}
Because $m\geq n$, we find $\partial_z^m  L_n^\alpha(z) 
= (-1)^m \delta_{m,n}$ with the Kronecker delta on the right-hand side. 
Then the remaining integral is nothing but the integral 
representation of $\Gamma(1+n+\alpha)$. \qed

If we replace the integral on the left-hand side of 
\eqref{eq:scalar_product_laguerre} with the generalized integral, 
formula \eqref{eq:scalar_product_laguerre} remains valid if 
$\alpha\in\cc\setminus-\nn$, because then the generalized integral 
is non-anomalous and can be computed by analytic continuation:

\begin{theorem}{Theorem}
\label{thm3.2}
Let $\alpha \in\cc\setminus-\nn$. Then
\eqref{eq:scalar_product_laguerre} remains true if we replace the 
usual integral with the generalized integral.
\end{theorem}

The case of negative integer $\alpha$ is more complicated.
The bilinear integral reduces to the standard integral also for negative integer 
$\alpha$ provided that both $m$ and $n$ are not less than $|\alpha|$:
\begin{theorem}{Theorem} \label{thm3.3}
Let $\alpha\in-\nn$ and $m,n\geq|\alpha|$. Then
\eqref{eq:scalar_product_laguerre} remains true (in the sense of usual
integrals):
  \begin{align}
\label{eq:scalar_product_laguerre_non_anomalous}
\int_0^\infty L_m^\alpha(z) L_n^\alpha(z) z^\alpha \e^{-z} \dd z 
= \frac{\Gamma(1+n+\alpha)}{n!}\delta_{m,n}=\frac{(n-|\alpha|)!}{n!}\delta_{m,n}
\end{align}
\end{theorem}

\proof
 By the identity \eqref{eq:special_alpha_Laguerre}
\begin{align}
 &\quad 
 \gen\int_0^\infty L_m^\alpha(z) L_n^\alpha(z) z^\alpha \e^{-z} \dd z
 \\ \notag 
 &=\frac{(m-|\alpha|)!(n-|\alpha|)!}{m!n!} \int_0^\infty  
L_{m-|\alpha|}^{|\alpha|}(z) 
L_{n-|\alpha|}^{|\alpha|}(z)  z^{|\alpha|} \e^{-z} \dd z
\\ \notag 
  &= \frac{(n-|\alpha|)!}{n!} \delta_{m,n}.
\end{align} \qed

For low degree Laguerre polynomials and  negative integer $\alpha$ we 
need anomalous  integrals:

\begin{theorem}{Theorem}
 \label{prop:anomalous_laguerre}
 Let $-\alpha \in \nn$, $m,n\in\nn_0$ such that 
 $|\alpha|>n$ and, without loss of generality, $m\geq n$. 
 Then 
 \begin{align}
 \label{eq:gen_Lint}
   \gen\int_0^\infty L_m^\alpha(z) L_n^\alpha(z) z^\alpha \e^{-z} \dd z 
   &= \frac{(-1)^{\alpha+n}}{n! (|\alpha|-n-1)!(m-n)}
         ,\qquad n<m;
   \\   \gen\int_0^\infty L_n^\alpha(z)^2 z^\alpha \e^{-z} \dd z 
   &=  \frac{(-1)^{\alpha+n+1}}{n! (|\alpha|-n-1)!}
  \psi(|\alpha|-n) . \end{align}
\end{theorem}

\begin{proof}
 Let us first consider the case $m>n$ and $|\alpha|>n$. We find 
 \begin{align}
  \gen\int_0^\infty L_m^\alpha(z) L_n^\alpha(z) z^\alpha \e^{-z} \dd z
  &= \frac{1}{m!} 
  \gen\int_0^\infty 
  \big(\partial^m \e^{-z} z^{m+\alpha}\big) L_n^\alpha(z) \dd z.
 \end{align}
We may integrate the right-hand side by parts $n+1$ times, each time 
collecting the regular value of the remaining integrand 
(cf. Appendix \ref{app:gen_int} for details). Thus 
\begin{align}
 & 
 \gen\int_0^\infty L_m^\alpha(z) L_n^\alpha(z) z^\alpha \e^{-z} \dd z 
 = \frac{1}{m!} \sum_{k=1}^{n+1} (-1)^k \rv 
 \Big( \big(\partial^{m-k} \e^{-z} z^{m+\alpha}\big) \partial^{k-1} L_n^{\alpha}(z) \Big).
\end{align}
Inserting the series expansions 
\begin{align} 
 \partial_z^{m-k} \e^{-z} z^{m+\alpha} &
 = (-1)^{m-k} \Bigg(\sum_{j=0}^{|\alpha|-m-1} +
  \sum_{j=|\alpha|-k}^\infty\Bigg) \frac{(-1)^j}{j!}
  (|\alpha|-m-j)_{m-k} z^{j+k-|\alpha|},\\ 
\partial_z^{k-1}  L_n^\alpha(z) 
&=  (-1)^n\sum_{i=0}^{n+1-k}\frac{(-\alpha-n)_{n-i-k+1}
                                       z^{i}}{i! (n-i-k+1)!},
\end{align}
we obtain a triple sum:  
\begin{align}\label{pukaw}
 &\quad\gen\int_0^\infty L_m^{\alpha}(z)  L_n^{\alpha}(z) z^{\alpha} \e^{-z} 
 \dd z
 \\ \notag 
 &= \frac{(-1)^{m+n}}{ m!}
  \sum_{k=1}^{n+1}  \rv 
 \Bigg( \Bigg(\sum_{j=0}^{|\alpha|-m-1} + \sum_{j=|\alpha|-k}^\infty\Bigg) \sum_{i=0}^{n+1-k}
 \frac{(-1)^j}{j!} (|\alpha|-m-j)_{m-k}  
 \\ \notag &\quad\times
\frac{(|\alpha|-n)_{n-k-i+1}}{i!(n-i-k+1)!} z^{i+k-|\alpha|+j}  \Bigg). 
\end{align}
The regular value evaluated at $0$ is then the coefficient for 
$j=|\alpha|-k-i$. One quickly finds that only the term $i=0$ is non-zero, 
so actually, the regular value is the coefficient corresponding to 
$j=|\alpha|-k$ and $i=0$. Hence we are left with the single sum
\begin{align}\eqref{pukaw}
&= \frac{(-1)^{m+n+|\alpha|-k}}{ m!}
  \sum_{k=1}^{n+1}   \frac{ (k-m)_{m-k}(|\alpha|-n)_{n-k+1} }{(|\alpha|-k)!(n+1-k)!} \\ \notag 
  &= \frac{(-1)^{n+\alpha}}{ m!(|\alpha|-n-1)!}
  \sum_{k=1}^{n+1}   \frac{ (m-k)! }{(n+1-k)!} \\ \notag 
&= \frac{(-1)^{n+\alpha}}{ n!(|\alpha|-n-1)! (m-n)},
\end{align}
where in the last step we used an  identity
about telescoping series recalled in Lemma \ref{telescope}.

Now consider $m=n<|\alpha|.$ We again use the Rodrigues type formula. 
However, we can now only integrate by parts $n$ times, so that a 
simple generalized integral survives: 
\begin{align}
 &\quad 
 \gen\int_0^\infty \big(L_n^{\alpha}(z)\big)^2 z^{\alpha} \e^{-z} \dd z
 \\ \notag 
 &= \frac{1}{n!} \Bigg( \sum_{k=1}^{n} (-1)^k \rv 
 \Big( \big(\partial^{n-k} \e^{-z} z^{n+\alpha}\big) \partial^{k-1} L_n^{\alpha}(z) \Big)
 + \gen\int_0^\infty \e^{-z} z^{n+\alpha} \dd z\Bigg).
\end{align}
The remaining generalized integral is the regularized Gamma function, 
which is treated as an example in \cite{DGR23a}:
\begin{align}
 \gen\int_0^\infty \e^{-z} z^{n+\alpha} \dd z
 = -\frac{(-1)^{\alpha+n}}{(|\alpha|-n-1)!} \psi(|\alpha|-n), 
 \quad n<|\alpha|\in\nn.
\end{align} 
The sum of regular values can be written as 
\begin{align} \label{pukaw2}
&\quad \frac{1}{ n!}
  \sum_{k=1}^{n}  \rv 
 \Bigg( \Bigg(\sum_{j=0}^{|\alpha|-n-1} + \sum_{j=|\alpha|-k}^\infty\Bigg) \sum_{i=0}^{n+1-k}
 \frac{(-1)^j}{j!}  (|\alpha|-n-j)_{n-k}
 \\ \notag &\quad\times 
\frac{(|\alpha|-n)_{1+n-k-i}}{i!(n-i-k+1)!} z^{i+k-|\alpha|+j}  \Bigg).
\end{align} 
In contrast to the case $m>n$, now there are two terms that contribute 
to $\rv$: $(j=|\alpha|-k \land i=0)$ and $(j=|\alpha|-n-1\land
i=n+1-k)$.
We obtain 
\begin{align}\eqref{pukaw2}=
  &\frac{(-1)^{n+\alpha}}{ n!(|\alpha|-n-1)!}
  \sum_{k=1}^n\Bigg(\frac{ (n-k)! }{(n+1-k)!}-   \frac{ (n-k)! }{(n+1-k)!} \Bigg)\,=\,0
\end{align} 
 We thus 
 proved \eqref{eq:gen_Lint}. 
\end{proof}

We may write down a ``generalized Gram matrix'' $G(\alpha)$ 
whose entries are the bilinear generalized integrals over Laguerre 
polynomials. We have 
\begin{align}
 G(\alpha) &= \diag\Big( \frac{\Gamma(1+\alpha+n)}{n!} \Big), 
 &&\quad \alpha\in\cc\setminus-\nn, 
 \end{align}
 and, for $\alpha\in-\nn$:
{\small
 \begin{align*}
 G(\alpha) &= 
     \bordermatrix{  & 0 & 1  & \cdots & |\alpha|-1 & |\alpha| \;\cdots\cr
       0  &- \frac{(-1)^{|\alpha|} \psi(|\alpha|)}{(|\alpha|-1)! 0! }  
        & \frac{(-1)^{|\alpha|}}{(|\alpha|-1)!  } 
        & \cdots
        & \frac{(-1)^{|\alpha|}}{(|\alpha|-1)(|\alpha|-1)!  }  & \cdots\cr
       1 & \frac{(-1)^{|\alpha|}}{(|\alpha|-1)!  } 
        & \frac{(-1)^{|\alpha|}\psi(|\alpha|-1)}{(|\alpha|-2)!  1!} 
        & \cdots 
        & \frac{(-1)^{1+|\alpha|}}{(|\alpha|-2)(|\alpha|-2)! } & \cdots \cr
      \vdots & \vdots
        & \vdots  
        & \ddots 
        & & \cr
      |\alpha|-1 & \frac{(-1)^{|\alpha|}}{(|\alpha|-1)(|\alpha|-1)! }  
      &  \frac{(-1)^{1+|\alpha|}}{(|\alpha|-2)(|\alpha|-2)! } 
      & & \frac{\psi(1)}{0!(|\alpha|-1)!}  & \cdots \cr  
       |\alpha|\cdots & \vdots & \vdots & 
       & \vdots  & \diag\Big( \frac{(n-|\alpha|)!}{n!} \Big) }.
\end{align*}}
In particular, all entries of the rows and columns $0,\dots,|\alpha|-1$ 
are non-zero and the sign of the entries of these rows and columns 
oscillates.

  Integration by parts seems to be the simplest way to determine
  bilinear generalized integrals of Laguerre polynomials.
  Alternatively, we can derive Thms \ref{thm3.2}, \ref{thm3.3}
  and
 \ref{prop:anomalous_laguerre}
from Thm \ref{prop:anomalous_tricomi}, using the fact that Laguerre
polynomials are essentially special cases of Tricomi functions. For
Thms
\ref{thm3.2} and \ref{thm3.3}
this is easy. To derive Thm  \ref{prop:anomalous_laguerre}
we need to analyze certain seemingly
singular expressions.

\begin{proof}[Alternative proof of Thm. \ref{prop:anomalous_laguerre}] 
 Let $m,n\in\nn_0$ and $\alpha\in\zz$. Let us show that 
 \begin{subequations}
 \begin{align} \label{eq:limit_diff.}
   &\gen\int_0^\infty L_n^\alpha(z)L_m^\alpha(z) \e^{-z}z^\alpha\d z\\\label{eq:limit_diff+} 
   = & \frac{(-1)^{m+n}}{m!n!} 
 \gen \int_0^\infty U_{-1-\alpha-2n,\alpha}(z) 
    U_{-1-\alpha-2m,\alpha}(z) \e^{-z}z^\alpha\d z,
    \quad m\neq n,
    \end{align}
    \end{subequations}
   and  
    \begin{subequations}
    \begin{align}
\label{eq:limit_same.}
&   \gen\int_0^\infty L_n^\alpha(z)^2\e^{-z}z^\alpha\d z\\&= \frac{1}{(n!)^2} 
 \gen \int_0^\infty U_{-1-\alpha-2n,\alpha}(z)^2\e^{-z}z^\alpha\d z. \label{eq:limit_same+}
   \end{align}
     \end{subequations} 

   The generalized integrals of Tricomi functions given by  Thm \ref{prop:anomalous_tricomi} are seemingly singular for the 
   parameters in \eqref{eq:limit_diff+} and \eqref{eq:limit_same+}.
   However, these singularities are merely apparent. To see this, 
   we replace  \eqref{eq:limit_diff+} and \eqref{eq:limit_same+} 
   by the following limits: 
 \begin{align} \label{eq:limit_diff}
 I_1 &\coloneqq \frac{(-1)^{m+n}}{m!n!} 
  \lim_{\epsilon\to0} \gen \int_0^\infty U_{-1-\alpha-2n-2\epsilon,\alpha}(z) 
    U_{-1-\alpha-2m-2\epsilon,\alpha}(z)\e^{-z}z^\alpha\d z ,
    \quad m\neq n,\\
\label{eq:limit_same}
I_2 &\coloneqq     \frac{1}{(n!)^2} 
  \lim_{\epsilon\to0} \gen   \int_0^\infty 
  U_{-1-\alpha-2n-2\epsilon,\alpha}(z)^2 
   \e^{-z}z^\alpha\d z .
 \end{align}
 
 To determine $I_1$, we may without loss of generality 
assume that $m>n$. Note that 
 \begin{align}
  &\quad \Bigg| \sum_{k=0}^{|\alpha|-1}
\big(\tfrac{1+\theta_1-|\alpha|}{2}\big)_k
\big(\tfrac{3+\theta_2-|\alpha|}{2}+k\big)_{|\alpha|-1-k}
\notag               \\&\times                                          
\Big(-\psi(|\alpha|-k)-\psi(k+1)
+\tfrac12H_k\big(\tfrac{1+\theta_1-|\alpha|}{2}\big)
- \tfrac12H_{|\alpha|-1-k}\big(\tfrac{3+\theta_2-|\alpha|}{2}+k\big)\Big)\Bigg|<\infty
 \end{align}
 for any values of the parameters because possible vanishing denominators 
 of the harmonic numbers are balanced by the Pochhammer symbol. Due to 
 the $\Gamma$ functions in the prefactor, see 
 Thm. \ref{prop:anomalous_tricomi}, this sum only contributes if 
 $-\alpha>m$ \emph{and} $-\alpha>n$. In the cases where the sum does 
 not contribute, we have 
 \begin{align}
  I_1 &=\frac{(-1)^{\alpha+m+n}}{2(m-n) n! m!}  
  \\ \notag 
  &\quad\times\lim_{\epsilon\to0}\begin{cases}
      \dfrac{\psi(-n-\epsilon)
    +\psi(-\alpha-n-\epsilon)
    }{\Gamma(-\alpha-n-\epsilon)
    \Gamma(-m-\epsilon)}
      - (m \leftrightarrow n),   &\quad \alpha\geq0;  \vspace{12pt} \\
     \dfrac{\psi(-n-\epsilon)
    +\psi(|\alpha|-n-\epsilon)
    }{\Gamma(-n-\epsilon)
    \Gamma(|\alpha|-m-\epsilon)}
      - (m \leftrightarrow n),    &\quad \alpha<0,\;m\geq|\alpha|.
        \end{cases}
 \end{align}
Now we use 
\begin{align}
 \frac{\psi(z)}{\Gamma(z)}\Big|_{z=-k} = (-1)^{k+1} k!
\end{align}
to obtain
\begin{align}
  I_1 &=
\begin{cases}
      0,   &\quad \alpha\geq0;  \vspace{12pt} \\
     \dfrac{(-1)^{\alpha+n}}{(m-n) n! \Gamma(|\alpha|-n)  }  
     ,    &\quad \alpha<0,\;m\geq|\alpha|.
        \end{cases}
\end{align}
In particular, if $\alpha<0$ and both $m,n\geq|\alpha|$, then $I_1=0$.

Let us now assume that $-\alpha>m$ and $-\alpha>n$ (and still, $m>n$). 
Then 
\begin{align}
\label{eq:I1_mn}
 I_1&= \frac{(-1)^{\alpha+m+n}}{2 n! m! \Gamma(|\alpha|-m) 
  \Gamma(|\alpha|-n)} 
  \Bigg( \frac{(-1)^m m! \Gamma(|\alpha|-m)-(-1)^n n!\Gamma(|\alpha|-n)}{m-n}
        \\ \notag 
    &\quad+
    \sum_{k=0}^{|\alpha|-1} (-n)_k (1-m+k)_{|\alpha|-1-k}
    \Big(  H_k(-n)- H_{|\alpha|-1-k}(1-m+k)\Big) \Bigg),
\end{align}
where 
\begin{align}
 (z)_k H_k(z) = \sum_{j=0}^{k-1} \prod_{l=0,l\neq j}^{k-1} (z+l).
\end{align}
With similar manipulations as in the proof of Theorem \ref{prop:anomalous_laguerre} and using Lemma \ref{telescope}, the sum 
in the last line  of \eqref{eq:I1_mn} can be simplified and we obtain 
the expression from Thm \ref{prop:anomalous_laguerre}.

The analysis of \eqref{eq:limit_same} is similar. We distinguish 
the three cases $\alpha>0$, $(\alpha<0 \land n\geq |\alpha|)$ and 
$-\alpha>n$. The sum in \eqref{eq:gen_Uint_2} only contributes if 
$-\alpha>n$. To obtain the limit, one needs to evaluate certain combinations 
of $\psi'(z)$, $\psi(z)$ and $\Gamma(z)$ at negative integers. We derive 
the respective formulas in Lemma \ref{lem:lem2}. 
\end{proof}

\appendix
\section{Generalized integral}
\label{app:gen_int}
\subsection{Definition}
Let us recall from \cite{DGR23a} the definition of the generalized 
integral where the integrand has a non-integrable 
homogeneous singularity at 0. 
\begin{definition}{Definition}
\label{def:one_sided_genInt}
We say that a function $f$ on $]0,\infty[$ is {\em integrable in 
the generalized sense} if it is integrable on $]1,\infty[$ and there 
exists a finite set $\Omega\subset\cc$ 
and complex coefficients $(f_k)_{k \in \Omega}$ such that 
\begin{align} 
f-\sum_{k\in\Omega} f_k r^k
\end{align} 
is integrable on $]0,1[$.  The set of such functions is denoted
$\cF$. For $f\in\cF$ we define
\begin{align}
\label{gener}
 \gen\int_0^\infty f(r)\dd r
 &\coloneqq \sum_{k\in\Omega\backslash\{-1\}}\frac{f_k}{k+1}+
\int_0^{1}\Big(f(r)-\sum_{k\in\Omega}f_k r^k\Big)\dd r+\int_{1}^\infty
f(r)\dd r.
\end{align}
\end{definition}
Note that the set  $\{k \in \Omega\, | \, \Re(k)\leq -1\}$ and the 
corresponding $f_k$ are uniquely determined by $f$. 
It is convenient to allow $k$ with $\Re(k)>-1$. The generalized 
integral of $f$ does not depend on the choice of $\Omega$.

It is clear that the generalized integral extends the standard integral: 
\begin{align}
\gen\int_0^\infty f(r)\dd r=\int_0^\infty f(r)\dd  r \quad\text{for }
  f \in L^1]0,\infty[  .
\end{align}

The generalized integral is in general coordinate dependent. 
As was proved in \cite{DGR23a}, the generalized integral
is invariant under scaling if and only if $f_{-1} = 0$, and invariant
under a large class of a change of variables if $f_k=0$ for every negative integer $k$. 

\begin{definition}{Definition}
The generalized integral \eqref{gener} is called {\em anomalous} 
if there exists $n \in \nn$ such that $f_{-n}\neq0$.
\end{definition}

\subsection{Integration by parts}

\begin{definition}{Definition}{Definition} We say that $F\in\cG$ if $F$ is a measurable function on
$[0,\infty[$, bounded on $[1,\infty[$ and there exists a finite set
$\Theta\subset\cc$ and complex coefficients $(F_k)_{k\in\Theta}$ and
$c\in\cc$, such that
\begin{align} F(r)-\sum_{k\in\Theta}F_kr^k-c\ln r\end{align}
has a limit as $r\to0$. For $F\in\cG$ we define the regular value of
$F$ at $0$:
\begin{align} \label{kaka}
\rv
F\coloneqq \lim_{r\searrow0}\Big(F(r)-\sum_{k\in\Theta\backslash\{0\}}F_kr^k-c\ln
r\Big).
\end{align}
\end{definition}

Note that given $F$ the coefficients $F_k$ for $k\in\Theta$ such 
that $\Re(k)\leq0,$ $k\neq0$, as well as $c$, are uniquely defined.  
Hence \eqref{kaka} depends only on $F$.

\begin{proposition}{Proposition} Let $f\in\cF$  and $r>0$. Then 
\begin{align} F(r)\coloneqq -\int_r^\infty f(y)\d y\end{align} 
belongs to $\cG$ and 
\begin{align}
  \gen\int_0^\infty f(y) \dd y 
  = -\rv F.\label{popi5}
 \end{align}
\end{proposition}

 \begin{proof}
 It is clear that $F$ is bounded on $[1,\infty[$ because 
 $f\in\cF$.
 Let $0<r<1$. Then
 \begin{align}
   \int_r^\infty f(y)\d y
&=  
\int_r^{1}\Big(f(y)-\sum_{k\in\Omega}f_k y^k\Big)\dd y+\int_{1}^\infty 
                            f(y)\dd y\\ \notag
&+
                                          \sum_{k\in\Omega\backslash\{-1\}}\frac{f_k}{k+1}(1-r^{k+1})-f_{-1}\ln r\\
   =&
\gen\int_0^\infty   f(y)\dd y-\int_0^r\Big(f(y)-\sum_{k\in\Omega}f_k y^k\Big)\dd y\label{puka1}                          \\ \notag
&-
                                          \sum_{k\in\Omega\backslash\{-1\}}\frac{f_k}{k+1}r^{k+1}-f_{-1}\ln r.
 \end{align}
The second term in \eqref{puka1} goes to zero as
$r\searrow0$. After pulling the last two terms to the left-hand side, 
we may thus perform the limit $r\searrow0$. We obtain  
\eqref{popi5}.
\end{proof}
\begin{corollary}
 Let $f$, $g$ be measurable functions and suppose that $fg',gf'\in\cF$, 
 and that $\lim_{r\to\infty}f(r)g(r)=0$. Then
  \begin{align}
  \gen\int_0^\infty f(r)g'(r)\dd r=
  -  \gen\int_0^\infty f'(r)g(r)\dd r-\rv(fg).\end{align}
  \end{corollary}

  \proof
Clearly, $(fg)'=fg'+f'g$, hence $(fg)'\in\cF$. Moreover
  \begin{align} f(r)g(r)=-\int_r^\infty (fg)'(y)\dd y
.\end{align} Therefore, it is enough to apply \eqref{popi5} to $F\coloneqq fg$. \qed

\subsection{Dimensional regularization}
\label{ssc:dim_reg}
We briefly recall from \cite{DGR23a} how the generalized 
integral can be computed using dimensional regularization. 

Let  $N\in\nn$ and let $f : \, ] 0 , \infty [ \times \{ \alpha \in \cc \, | \, \Re(\alpha) > -N-1 \} \to \cc$ be a function such that $f(r , \cdot)$ is holomorphic for each $r$, $\| f(\cdot, \alpha) \|_{L^1[1, \infty[}$ is bounded locally uniformly in $\alpha$, and there exist holomorphic functions $f_0, \dots, f_N$ of $\alpha$ such that the $L^1]0,1]$ norm of $f(r,\alpha)-\sum_{n=0}^N r^{\alpha+n}f_n(\alpha)$ is bounded locally uniformly in $\alpha$. Then $f(\cdot, \alpha)$ is integrable in the generalized sense, and for $- \alpha \not \in \{ 1 , \dots, N \}$ one has
\begin{align} 
\label{anla}
&\quad\gen\int_0^\infty f(r,\alpha)\dd r
\\ \notag &=   \sum_{n=0}^N \frac{f_{n}(\alpha)}{\alpha+n+1}
+ \int_0^1 \Big( f(r,\alpha) - \sum_{n=0}^N r^{\alpha+n} f_n(\alpha) \Big) \dd r
+ \int_1^\infty f(r,\alpha) \dd r.
\end{align}
By Morera's theorem, the right hand side is, away from the poles at $-1, \dots, -N$, a holomorphic function of $\alpha$. Therefore, to obtain \eqref{anla} in the non-anomalous case it is enough to compute \eqref{anla} in
the region where the usual integral is convergent and continue analytically. 

Let $m \in \{ 1 , \dots, N \}$. The right hand side of \eqref{anla} has 
a simple pole at $\alpha = -m$ with residue $f_{m-1}(-m)$ (possibly zero). 
Its finite part is 
\begin{equation}
\operatornamewithlimits{fp}_{\alpha \to -m} \gen \int_0^\infty f(r,\alpha)\dd r
  =\lim_{\alpha\to-m}\Bigg(\gen\int_0^\infty f(r,\alpha)\dd
r-\frac{f_{m-1}(-m)}{\alpha+m}\Bigg).  
\end{equation}
To recover the generalized integral at $\alpha=-m$, 
one also needs to subract a finite term: 
\begin{align}
 \gen\int_0^\infty f(r,-m)\dd r = \operatornamewithlimits{fp}_{\alpha \to -m} \gen \int f(r,\alpha)\dd r -f'_{m-1}(-m). 
 \label{eq:dimreg_anomalous}
\end{align}

\section{Two useful lemmas}
\label{app:lemmas}
We prove several identities that have been used in the computation 
of generalized integrals:

 \begin{lemma}{Lemma}
 \label{telescope}
We have 
 \begin{align} \label{telescope1}
 \sum_{k=0}^n\frac{(a)_k}{k!}=\frac{(a+1)_n}{n!},
\end{align}
which implies for $m,n\in\nn$, $m>n$:
\begin{align}
 \sum_{k=1}^{n+1} \frac{(m-k)!}{(n+1-k)!} = \frac{m!}{n!} \frac{1}{m-n}.
\end{align}
\end{lemma}

\begin{proof}
The first identity follows from a standard telescoping sum argument. The second identity follows from the first by setting $a=m-n$.
\end{proof}

\begin{lemma}{Lemma}
 \label{lem:lem2} For $n\in\nn_0$, we have 
\begin{align}
 \frac{\psi'(z)}{\Gamma(z)^2}\Big|_{z=-n} =(n!)^2,\quad \text{and}\quad 
   \frac{\psi'(z)-\psi(z)^2}{\Gamma(z)} \Big|_{z=-n}
  &= (-1)^n 2 n! \psi(1+n).
\end{align}
\end{lemma}
\begin{proof}
  Let $z\in\cc\setminus-\nn_0$. It is easy to verify that 
  for $n\in\nn$, we have  
 \begin{align}
  \frac{\psi'(z)}{\Gamma(z)^2} 
  = \big((z)_{n}\big)^2  \frac{\psi'(z+n)}{\Gamma(z+n)^2}
  + \sum_{j=1}^n  \frac{\big((z)_{j-1}\big)^2}{\Gamma(z+j)^2}.
 \end{align}
 The first identity now follows from taking the limit $z\to-n$.
The proof of the second identity works similarly. We first check that 
\begin{align}
 \frac{\psi'(z)-\psi(z)^2}{\Gamma(z)} 
 &= (z)_n \frac{\psi'(z+n)-\psi(z+n)^2}{\Gamma(z+n)}
 + 2 \sum_{j=1}^{n} (z)_{j-1} \frac{\psi(z+j)}{\Gamma(z+j)}.
\end{align}
The claim then follows by taking $z\to-n$, using 
\begin{align} \label{eq:zero123}
  \lim_{\epsilon\to 0} \frac{\psi'(\epsilon)-\psi(\epsilon)^2}{\Gamma(\epsilon)}= 2\psi(1)
\end{align}
and standard identities satisfied by $\psi(z)$.
\end{proof}

\end{document}